\documentclass[12pt,leqno]{article}
\usepackage{amssymb,amsthm,amsmath,latexsym}
\newtheorem{thm}{Theorem}

\newtheorem{lem}[thm]{Lemma}

\theoremstyle{remark}
\newtheorem{rem}[thm]{Remark}

\DeclareMathOperator{\Fix}{Fix}
\DeclareMathOperator{\SL}{SL}

\begin{document}

\title{On the Classification of Extremal Doubly Even
Self-Dual Codes with $2$-Transitive Automorphism Groups}

\author{
Naoki Chigira\thanks{
Department of Mathematics, 
Kumamoto University,
Kumamoto 860--8555, Japan. 
email: chigira@kumamoto-u.ac.jp},
Masaaki Harada\thanks{
Department of Mathematical Sciences,
Yamagata University,
Yamagata 990--8560, Japan. 
email: mharada@sci.kj.yamagata-u.ac.jp}
and 
Masaaki Kitazume\thanks{
Department of Mathematics and Informatics,
Chiba University,
Chiba 263--8522, Japan.
email: kitazume@math.s.chiba-u.ac.jp}
}

\maketitle

\begin{abstract}
In this note, we complete the classification of extremal
doubly even self-dual codes with $2$-transitive automorphism
groups.
\end{abstract}



\section{Introduction}

As described in~\cite{RS-Handbook},
self-dual codes are an important class of linear codes for both
theoretical and practical reasons.
It is a fundamental problem to classify self-dual codes
of modest lengths 
and determine the largest minimum weight among self-dual codes
of that length (see~\cite{Huffman05, RS-Handbook}).
It was shown in~\cite{MS73} that
the minimum weight $d$ of a  doubly even
self-dual code of length $n$ is bounded by
$d  \le 4  \lfloor{\frac {n}{24}} \rfloor + 4$.
A doubly even self-dual code meeting the bound is called  {\em extremal}.
A common strategy for the problem whether there is an
extremal doubly even self-dual code for a given length
is to classify extremal doubly even
self-dual codes with a given nontrivial automorphism group
(see~\cite{Huffman05, RS-Handbook}).
Recently, Malevich and Willems~\cite{MW} have shown that
if $C$ is an extremal doubly even self-dual code with a $2$-transitive
automorphism group then
$C$ is equivalent to one of the extended quadratic residue codes of
lengths $8,24,32$, $48,80,104$, the second-order Reed--Muller
code of length $32$ or a putative
extremal doubly even self-dual code of length $1024$
invariant under the group $T \rtimes \SL(2,2^5)$, where
$T$ is an elementary abelian group of order $1024$.

The aim of this note is to complete the classification
of extremal doubly even self-dual codes with $2$-transitive
automorphism groups.
This is completed by excluding the open case in the above 
characterization~\cite{MW}, using Theorem~A in~\cite{CHK}.

\begin{thm}\label{main}
Let $C$ be an extremal doubly even self-dual code with a $2$-transitive
automorphism group.
Then $C$ is  equivalent to one of the the extended quadratic 
residue codes of lengths $8,24,32,48,80,104$ or
the second-order Reed--Muller code of length $32$.
\end{thm}

\section{Proof of Theorem~\ref{main}}

For an $n$-element set $\Omega$, 
the power set ${\cal P}(\Omega)$ 
-- the family of all subsets of $\Omega$ --
is regarded as an $n$-dimensional binary vector space 
with the inner product $(X, Y) \equiv |X \cap Y| \pmod 2$
for $X,Y \in {\cal P}(\Omega)$.
The {\em weight} of $X$ is defined to be the integer $|X|$. 
A subspace $C$ of ${\cal P}(\Omega)$ 
is called a {\em code} of length $n$. 
Note that all codes in this note are binary.
The {\em dual code} $C^\perp$ of $C$ is 
the set of all $X \in {\cal P}(\Omega)$ 
satisfying $(X, Y)=0$ for all $Y\in C$. 
A code $C$ is said to be 
{\em self-orthogonal} if $C \subset C^\perp$, and 
{\em self-dual} if $C = C^\perp$. 
A {\em doubly even} code is a code 
whose codewords have weight a multiple of $4$.

Let $G$ be a 
permutation group on an $n$-element set $\Omega$.
We define the code $C(G,\Omega)$ by
\[
C(G,\Omega)= 
\langle \Fix(\sigma) \mid \sigma \in I(G)\rangle^\perp,
\]
where 
$I(G)$ denotes the set of involutions of $G$
and $\Fix(\sigma)$ is the set of fixed points of $\sigma$ on $\Omega$.

\begin{thm}[Chigira, Harada and Kitazume~\cite{CHK}]
\label{thm:CHK}
Let $C$ be a binary self-orthogonal code of length $n$
invariant under the group $G$.
Then $C \subset C(G,\Omega)$.
\end{thm}

By using Theorem~\ref{thm:CHK},
some self-dual codes invariant under sporadic almost 
simple groups were constructed in~\cite{CHK}.
In this note, we apply Theorem~\ref{thm:CHK} to 
a family of 2-transitive groups containing the group
$(2^{10}) \rtimes \SL(2,2^5)$. 

Let $r, s$ be positive integers. 
We consider the following group $G$ 
$$ G = T \rtimes H \quad 
(T = (2^r)^{2s}, H = \SL(2s, 2^r)),  $$
where the group $T$ is regarded as the natural module $GF(2^r)^{2s}$ of $H$. 
Here $T$ acts regularly on $T$ itself and $H$ acts on $T$ as
the stabilizer of the unit of $T$, which is regarded as
the zero vector of $GF(2^r)^{2s}$.
Then $G$ naturally acts 2-transitively on $T$. 

\begin{lem}
There is no self-dual code of length $2^{2rs}$
invariant under $G = T \rtimes H$. 
\end{lem}
\begin{proof}
By the fundamental theory of Jordan canonical forms in
basic linear algebra,
the dimension of the subspace of $GF(2^r)^{2s}$
spanned by the vectors fixed by
an involution in $H = \SL(2s, 2^r)$ is equal to or greater than $s$.
Then it is easily seen that
there exist two involutions $\sigma, \tau$ in $H$
such that each of them fixes some $s$-dimensional subspace
of $GF(2^r)^{2s}$, and
the zero vector is the only vector fixed  by both of them
(i.e.\ $T = \Fix(\sigma)\oplus\Fix(\tau)$).
As codewords in $C(G,\Omega)^\perp$, the inner product
$(\Fix(\sigma), \Fix(\tau))$ is equal to $1$,
since $|\Fix(\sigma)\cap \Fix(\tau)|=1$.
This yields that $C(G, T)^\perp$ is not self-orthogonal.

Suppose that 
$B$ is a self-dual code invariant under $G$.
By Theorem~\ref{thm:CHK}, $B \subset C(G, T)$.
Since $B^\perp \supset C(G, T)^\perp$ and $B=B^\perp$,
$C(G, T)^\perp$ is self-orthogonal.
This is a contradiction.
\end{proof}

The case $(r,s)=(5,1)$ in the above lemma completes the 
proof of Theorem~\ref{main}.

\begin{rem}
In the above proof, the cardinality of the fixed subspace of dimension $s$ 
is $2^{rs}$, which is smaller than the value
$ 4  \lfloor{\frac {2^{2rs}}{24}} \rfloor + 4$,
except for the cases $(r, s) = (1, 2), (2, 1)$. 
This shows immediately that 
there is no extremal doubly even self-dual code of length $2^{2rs}$ 
invariant under the group $G = T \rtimes \SL(2s, 2^r)$
if $rs >2$. 

On the other hand, 
the smallest cardinality of the fixed subspace of an involution
in $\SL(2s-1, 2^r)$ is $2^{rs}$.
If $s >1$ then this number is smaller than the value 
$ 4  \lfloor{\frac {2^{(2s-1)r}}{24}} \rfloor + 4$,
except for the small cases
$(r, s) = (1, 2), (1,3), (2, 2)$.
When $(r, s) = (1, 2)$ or $(1,3)$, the code $C(G, T)$,
for $G = T \rtimes \SL(2s-1, 2^r)$ where $T= (2^r)^{2s-1}$,
is equivalent to
the extended Hamming code of length $8$, or
the second-order Reed--Muller code of length $32$
(see~\cite[Example 2.10]{CHK}), respectively.
For the remaining case $(r, s) = (2, 2)$
(i.e.\ $G = T \rtimes \SL(3, 2^2)$, $T=2^6$),
the smallest cardinality of the fixed subspace of an involution
is $16\ (>12)$,
and so such an argument does not work.
(Indeed the code $C(G, T)^\perp$ is self-orthogonal
with minimum weight $16$.)
\end{rem}

\bigskip
\noindent {\bf Acknowledgment.}
This work is supported by 
JSPS KAKENHI Grant Numbers 23340021, 24340002, 24540024.


\end{document}